\documentclass{amsart}
\usepackage{amsfonts}

\newtheorem{theorem}{Theorem}
\theoremstyle{plain}

\newtheorem{corollary}{Corollary}

\newtheorem{definition}{Definition}

\newtheorem{remark}{Remark}

\numberwithin{equation}{section}
\input{tcilatex}

\begin{document}
\title[On some inequalities of Hadamard Type]{On some inequalities for $m$-
and $\left( \alpha ,m\right) $-logarithmically convex functions}
\author{Mevl\"{u}t TUN\c{C}}
\address{Kilis 7 Aral\i k University, Faculty of Science and Arts,
Department of Mathematics, Kilis, 79000, Turkey.}
\email{mevluttunc@kilis.edu.tr}

\begin{abstract}
In this paper, we establish some new integral inequalities for for $m$- and $%
\left( \alpha ,m\right) $-logarithmically convex functions.
\end{abstract}

\maketitle

\section{Introduction}

In \cite{bai}, the concepts of $m$- and $\left( \alpha ,m\right) $%
-logarithmically convex functions were introduced as follows.

\begin{definition}
\label{d}A function $f:[0,b]\rightarrow (0,\infty )$ is said to be $m$%
-logarithmically convex if the inequality%
\begin{equation}
f\left( tx+m\left( 1-t\right) y\right) \leq \left[ f\left( x\right) \right]
^{t}\left[ f\left( y\right) \right] ^{m\left( 1-t\right) }  \label{d1}
\end{equation}%
holds for all $x,y\in \lbrack 0,b]$, $m\in (0,1]$, and $t\in \lbrack 0,1]$.
\end{definition}

Obviously, if putting $m=1$ in Definition \ref{d}, then $f$ is just the
ordinary logarithmically convex on $\left[ 0,b\right] $.

\begin{definition}
\label{dd}A function $f:[0,b]\rightarrow (0,\infty )$ is said to be $\left(
\alpha ,m\right) $-logarithmically convex if%
\begin{equation}
f\left( tx+m\left( 1-t\right) y\right) \leq \left[ f\left( x\right) \right]
^{t^{\alpha }}\left[ f\left( y\right) \right] ^{m\left( 1-t^{\alpha }\right)
}  \label{d2}
\end{equation}%
holds for all $x,y\in \lbrack 0,b]$, $\left( \alpha ,m\right) \in \left( 0,1%
\right] \times \left( 0,1\right] ,$ and $t\in \lbrack 0,1]$.
\end{definition}

Clearly, when taking $\alpha =1$ in Definition \ref{dd}, then $f$ becomes
the standard $m$-logarithmically convex function on $\left[ 0,b\right] $.

In \cite{dr}, authors proved that the following inequalities of
Hermite-Hadamard type hold for $log$-convex functions:

\begin{theorem}
Let $f:I\rightarrow \left[ 0,\infty \right) $ be a log-convex mapping on $I$
and $a,b\in I$ with $a<b$. Then one has the inequality:%
\begin{equation}
f\left( A\left( a,b\right) \right) \leq \frac{1}{b-a}\int_{a}^{b}G\left(
f\left( x\right) ,f\left( a+b-x\right) \right) dx\leq G\left( f\left(
a\right) ,f\left( b\right) \right) .  \label{dr1}
\end{equation}
\end{theorem}

\begin{theorem}
\bigskip Let $f:I\rightarrow \left( 0,\infty \right) $ be a log-convex
mapping on $I$ and $a,b\in I$ with $a<b$. Then one has the inequality:%
\begin{eqnarray}
f\left( \frac{a+b}{2}\right) &\leq &\exp \left[ \frac{1}{b-a}\int_{a}^{b}\ln %
\left[ f\left( x\right) \right] dx\right]  \label{dr2} \\
&\leq &\frac{1}{b-a}\int_{a}^{b}G\left( f\left( x\right) ,f\left(
a+b-x\right) \right) dx\leq \frac{1}{b-a}\int_{a}^{b}f\left( x\right) dx 
\notag \\
&\leq &L\left( f\left( a\right) ,f\left( b\right) \right) \leq \frac{f\left(
a\right) +f\left( b\right) }{2}  \notag
\end{eqnarray}%
where $G(p,q):=\sqrt{pq}$ is the geometric mean and $L(p,q):=\frac{p\text{-}q%
}{\ln p-\ln q}$ $(p\neq q)$ is the logarithmic mean of the strictly positive
real numbers $p,q,$ i.e., 
\begin{equation*}
L\left( p,q\right) =\frac{p-q}{\ln p-\ln q}\text{ }if\text{ }p\neq q\text{
and }L(p,p)=p.
\end{equation*}
\end{theorem}

The main purpose of this paper is to prove some inequalities of Hadamard
type for $m$- and $\left( \alpha ,m\right) $-logarithmically convex
functions. Also we give some results for logarithmically convex functions.

\section{mean result}

We now consider the following means will be used in this paper.

a) The arithmetic mean:%
\begin{equation*}
A=A\left( a,b\right) :=\frac{a+b}{2},\text{\ \ }a,b\geq 0,
\end{equation*}

b) The geometric mean: 
\begin{equation*}
G=G\left( a,b\right) :=\sqrt{ab},\text{ \ }a,b\geq 0,
\end{equation*}

c) The logarithmic mean:

\begin{equation*}
L=L\left( a,b\right) :=\left\{ 
\begin{array}{l}
a\text{ \ \ \ \ \ \ \ \ \ \ \ \ \ if \ \ }a=b \\ 
\frac{b-a}{\ln b-\ln a}\text{ \ \ \ \ \ if \ \ }a\neq b%
\end{array}%
\right. ,\text{ \ }a,b\geq 0.
\end{equation*}

In this section, some Hadamard type inequalities for $m$- and $\left( \alpha
,m\right) $-logarithmically convex functions will be given.

\begin{theorem}
Let $f:\left[ 0,\infty \right) \rightarrow \left( 0,\infty \right) $ be a $m$%
-logarithmically convex function on $\left[ 0,\frac{b}{m}\right] $ with $%
a,b\in \left[ 0,\infty \right) ,$ $a<b,$ $m\in \left( 0,1\right] .$ Then%
\begin{equation}
\frac{1}{b-a}\int_{a}^{b}f\left( x\right) dx\leq \min \{L\left( f\left(
a\right) ,f\left( \frac{b}{m}\right) ^{m}\right) ,L\left( f\left( b\right)
,f\left( \frac{a}{m}\right) ^{m}\right) \},  \label{4}
\end{equation}%
where $L\left( a,b\right) $ is logarithmic mean.
\end{theorem}

\begin{proof}
Since $f$ is $m$-logarithmically convex function on $\left[ 0,\frac{b}{m}%
\right] $, we have that 
\begin{equation}
f\left( ta+m\left( 1-t\right) \frac{b}{m}\right) \leq \left[ f\left(
a\right) \right] ^{t}\left[ f\left( \frac{b}{m}\right) \right] ^{m\left(
1-t\right) }  \label{m}
\end{equation}%
and%
\begin{equation}
f\left( tb+m\left( 1-t\right) \frac{a}{m}\right) \leq \left[ f\left(
b\right) \right] ^{t}\left[ f\left( \frac{a}{m}\right) \right] ^{m\left(
1-t\right) }  \label{x}
\end{equation}%
for all $t\in \left[ 0,1\right] .$ By integrating the resulting inequality
on $\left[ 0,1\right] $ with respect to $t$, we get%
\begin{eqnarray*}
\int_{0}^{1}f\left( ta+\left( 1-t\right) b\right) dt &\leq &f\left( \frac{b}{%
m}\right) ^{m}\int_{0}^{1}\left( \frac{f\left( a\right) }{f\left( \frac{b}{m}%
\right) }\right) ^{t}dt, \\
\int_{0}^{1}f\left( tb+\left( 1-t\right) a\right) dt &\leq &f\left( \frac{a}{%
m}\right) ^{m}\int_{0}^{1}\left( \frac{f\left( b\right) }{f\left( \frac{a}{m}%
\right) }\right) ^{t}dt.
\end{eqnarray*}%
However,%
\begin{eqnarray*}
\int_{0}^{1}\left( \frac{f\left( a\right) }{f\left( \frac{b}{m}\right) }%
\right) ^{t}dt &=&\frac{\frac{f\left( a\right) }{f\left( \frac{b}{m}\right)
^{m}}-1}{\ln f\left( a\right) -\ln f\left( \frac{b}{m}\right) ^{m}} \\
&=&\frac{f\left( a\right) -f\left( \frac{b}{m}\right) ^{m}}{f\left( \frac{b}{%
m}\right) ^{m}\left[ \ln f\left( a\right) -\ln f\left( \frac{b}{m}\right)
^{m}\right] } \\
&=&\frac{1}{f\left( \frac{b}{m}\right) ^{m}}L\left( f\left( a\right)
,f\left( \frac{b}{m}\right) ^{m}\right) ,
\end{eqnarray*}%
\begin{equation*}
\int_{0}^{1}\left( \frac{f\left( a\right) }{f\left( \frac{b}{m}\right) ^{m}}%
\right) ^{t}dt=\frac{1}{f\left( \frac{a}{m}\right) ^{m}}L\left( f\left(
b\right) ,f\left( \frac{a}{m}\right) ^{m}\right)
\end{equation*}%
and%
\begin{equation*}
\int_{0}^{1}f\left( ta+\left( 1-t\right) b\right) dt=\int_{0}^{1}f\left(
tb+\left( 1-t\right) a\right) dt=\frac{1}{b-a}\int_{a}^{b}f\left( x\right) dx
\end{equation*}%
and the inequality (\ref{4}) is obtained. Which completes the proof.
\end{proof}

\begin{theorem}
Let $f:\left[ 0,\infty \right) \rightarrow \left( 0,\infty \right) $ be a $m$%
-logarithmically convex function on $\left[ 0,\frac{b}{m}\right] $ with $%
a,b\in \left[ 0,\infty \right) ,$ $a<b,$ $m\in \left( 0,1\right] .$ Then%
\begin{equation}
f\left( \frac{a+b}{2}\right) \leq \frac{1}{b-a}\int_{a}^{b}G\left( f\left(
x\right) ,\left[ f\left( \frac{a+b-x}{m}\right) \right] ^{m}\right)
\label{11}
\end{equation}%
and%
\begin{equation}
\frac{1}{b-a}\int_{0}^{1}G\left( f\left( x\right) ,f\left( a+b-x\right)
\right) dx\leq L\left( f\left( a\right) f\left( b\right) ,\left[ f\left( 
\frac{a}{m}\right) f\left( \frac{b}{m}\right) \right] ^{m}\right)  \label{22}
\end{equation}%
where $G\left( a,b\right) ,$ $L\left( a,b\right) $ are geometric,
logarithmic mean respectively.
\end{theorem}

\begin{proof}
Since $f$ is $m$-log-convex, we have that%
\begin{eqnarray}
f\left( ta+m\left( 1-t\right) \frac{b}{m}\right) &\leq &\left[ f\left(
a\right) \right] ^{t}\left[ f\left( \frac{b}{m}\right) \right] ^{m\left(
1-t\right) },  \label{p} \\
f\left( tb+m\left( 1-t\right) \frac{a}{m}\right) &\leq &\left[ f\left(
b\right) \right] ^{t}\left[ f\left( \frac{a}{m}\right) \right] ^{m\left(
1-t\right) }  \notag
\end{eqnarray}%
for all $t\in \left[ 0,1\right] $. If we multiply the above inequalities and
take square roots, we obtain%
\begin{equation*}
G\left( f\left( ta+\left( 1-t\right) b\right) ,f\left( tb+\left( 1-t\right)
a\right) \right) \leq \left[ f\left( a\right) f\left( b\right) \right] ^{t}%
\left[ f\left( \frac{a}{m}\right) f\left( \frac{b}{m}\right) \right]
^{m\left( 1-t\right) }
\end{equation*}%
for all $t\in \left[ 0,1\right] $. Integrating this inequality on $[0,1]$
with respect to $t$, we obtain%
\begin{eqnarray*}
&&\int_{0}^{1}G\left( f\left( ta+\left( 1-t\right) b\right) ,f\left(
tb+\left( 1-t\right) a\right) \right) dt \\
&\leq &\left[ f\left( \frac{a}{m}\right) f\left( \frac{b}{m}\right) \right]
^{m}\int_{0}^{1}\left[ \frac{f\left( a\right) f\left( b\right) }{\left[
f\left( \frac{a}{m}\right) f\left( \frac{b}{m}\right) \right] ^{m}}\right]
^{t}dt \\
&=&\left[ f\left( \frac{a}{m}\right) f\left( \frac{b}{m}\right) \right] ^{m}%
\frac{\frac{f\left( a\right) f\left( b\right) }{\left[ f\left( \frac{a}{m}%
\right) f\left( \frac{b}{m}\right) \right] ^{m}}-1}{\ln \frac{f\left(
a\right) f\left( b\right) }{\left[ f\left( \frac{a}{m}\right) f\left( \frac{b%
}{m}\right) \right] ^{m}}} \\
&=&\frac{f\left( a\right) f\left( b\right) -\left[ f\left( \frac{a}{m}%
\right) f\left( \frac{b}{m}\right) \right] ^{m}}{\ln f\left( a\right)
f\left( b\right) -\ln \left[ f\left( \frac{a}{m}\right) f\left( \frac{b}{m}%
\right) \right] ^{m}} \\
&=&L\left( f\left( a\right) f\left( b\right) ,\left[ f\left( \frac{a}{m}%
\right) f\left( \frac{b}{m}\right) \right] ^{m}\right)
\end{eqnarray*}%
If we change the variable $x=ta+\left( 1-t\right) b$, $t\in \left[ 0,1\right]
$, we get%
\begin{equation*}
\int_{0}^{1}G\left( f\left( ta+\left( 1-t\right) b\right) ,f\left( tb+\left(
1-t\right) a\right) \right) dt=\frac{1}{b-a}\int_{0}^{1}G\left( f\left(
x\right) ,f\left( a+b-x\right) \right) dx
\end{equation*}%
and the inequality in (\ref{22}) is proved.

By (\ref{p}), for $t=\frac{1}{2},$ we have that%
\begin{equation*}
f\left( \frac{x+y}{2}\right) \leq \sqrt{\left[ f\left( x\right) \right] %
\left[ f\left( \frac{y}{m}\right) \right] ^{m}}
\end{equation*}%
for all $x,y\in \left[ 0,\infty \right) $. If we choose $x=ta+\left(
1-t\right) b$, $y=tb+\left( 1-t\right) a$, we get the inequality%
\begin{equation*}
f\left( \frac{a+b}{2}\right) \leq \sqrt{\left[ f\left( ta+\left( 1-t\right)
b\right) \right] \left[ f\left( \frac{ta+\left( 1-t\right) b}{m}\right) %
\right] ^{m}}
\end{equation*}%
for all $t\in \left[ 0,1\right] $. Integrating this inequality on $\left[ 0,1%
\right] $ over $t$, we obtain the inequality in (\ref{11}) . This completes
the proof of the theorem.
\end{proof}

\begin{remark}
If we take $m=1$ in inequality (\ref{11}) and (\ref{22}), we obtain one
inequality such that special version of inequality (\ref{dr1}).
\end{remark}

\begin{theorem}
Let $f:\left[ 0,\infty \right) \rightarrow \left( 0,\infty \right) $ be an $%
\left( \alpha ,m\right) $-logarithmically convex function on $\left[ 0,\frac{%
b}{m}\right] $ with $a,b\in \left[ 0,\infty \right) ,$ $0\leq a<b<\infty ,$ $%
\left( \alpha ,m\right) \in \left( 0,1\right] \times \left( 0,1\right] .$
Then%
\begin{equation}
\frac{1}{b-a}\int_{a}^{b}f\left( x\right) dx\leq \min \{f\left( \frac{b}{m}%
\right) ^{m}M\left( \alpha \right) ,\text{ }f\left( \frac{a}{m}\right)
^{m}T\left( \alpha \right) \},  \label{31}
\end{equation}%
where%
\begin{equation*}
\begin{array}{cc}
\varphi =\frac{f\left( a\right) }{f\left( \frac{b}{m}\right) ^{m}}, & \ell =%
\frac{f\left( b\right) }{f\left( \frac{a}{m}\right) ^{m}}%
\end{array}%
\end{equation*}
and%
\begin{equation}
M\left( \alpha \right) =\left\{ 
\begin{array}{cc}
1 & ,\varphi =1 \\ 
\frac{\varphi ^{\alpha }-1}{\alpha \ln \varphi } & ,0<\varphi <1%
\end{array}%
\right. \text{ and }T\left( \alpha \right) =\left\{ 
\begin{array}{cc}
1 & ,\ell =1 \\ 
\frac{\ell ^{\alpha }-1}{\alpha \ln \ell } & ,0<\ell <1%
\end{array}%
\right.  \label{34}
\end{equation}
\end{theorem}

\begin{proof}
Since $f$ is an $\left( \alpha ,m\right) $-logarithmically convex function
on $\left[ 0,\frac{b}{m}\right] $, we have that 
\begin{equation*}
f\left( ta+m\left( 1-t\right) \frac{b}{m}\right) \leq \left[ f\left(
a\right) \right] ^{t^{\alpha }}\left[ f\left( \frac{b}{m}\right) \right]
^{m\left( 1-t^{\alpha }\right) }
\end{equation*}%
and%
\begin{equation*}
f\left( tb+m\left( 1-t\right) \frac{a}{m}\right) \leq \left[ f\left(
b\right) \right] ^{t^{\alpha }}\left[ f\left( \frac{a}{m}\right) \right]
^{m\left( 1-t^{\alpha }\right) }
\end{equation*}%
for all $t\in \left[ 0,1\right] .$ Integrating the above inequality on $%
[0,1] $ with respect to $t$, we get%
\begin{eqnarray*}
\int_{0}^{1}f\left( ta+m\left( 1-t\right) \frac{b}{m}\right) dt &\leq
&f\left( \frac{b}{m}\right) ^{m}\int_{0}^{1}\left( \frac{f\left( a\right) }{%
f\left( \frac{b}{m}\right) }\right) ^{t^{\alpha }}dt, \\
\int_{0}^{1}f\left( tb+m\left( 1-t\right) \frac{a}{m}\right) dt &\leq
&f\left( \frac{a}{m}\right) ^{m}\int_{0}^{1}\left( \frac{f\left( b\right) }{%
f\left( \frac{a}{m}\right) }\right) ^{t^{\alpha }}dt.
\end{eqnarray*}%
When $\varphi =1$, we have 
\begin{equation*}
\int_{0}^{1}\varphi ^{t^{\alpha }}dt=1.
\end{equation*}
When $\varphi <1$, we have%
\begin{equation*}
\int_{0}^{1}\varphi ^{t^{\alpha }}dt\leq \int_{0}^{1}\varphi ^{\alpha t}dt=%
\frac{\varphi ^{\alpha }-1}{\alpha \ln \varphi }.
\end{equation*}%
When $\ell =1$, we have 
\begin{equation*}
\int_{0}^{1}\ell ^{t^{\alpha }}dt=1.
\end{equation*}
When $\ell <1$, we have%
\begin{equation*}
\int_{0}^{1}\ell ^{t^{\alpha }}dt\leq \int_{0}^{1}\ell ^{\alpha t}dt=\frac{%
\ell ^{\alpha }-1}{\alpha \ln \ell }.
\end{equation*}%
and%
\begin{equation*}
\int_{0}^{1}f\left( ta+\left( 1-t\right) b\right) dt=\int_{0}^{1}f\left(
tb+\left( 1-t\right) a\right) dt=\frac{1}{b-a}\int_{a}^{b}f\left( x\right) dx
\end{equation*}%
and the inequality (\ref{31}) is obtained. Which is required.
\end{proof}

\begin{corollary}
Let $f:\left[ 0,\infty \right) \rightarrow \left( 0,\infty \right) $ be an $%
m $-logarithmically convex function on $\left[ 0,\frac{b}{m}\right] $ with $%
a,b\in \left[ 0,\infty \right) ,$ $0\leq a<b<\infty ,$ $m\in \left( 0,1%
\right] .$ Then%
\begin{equation}
\frac{1}{b-a}\int_{a}^{b}f\left( x\right) dx\leq \min \{f\left( \frac{b}{m}%
\right) ^{m}M\left( 1\right) ,\text{ }f\left( \frac{a}{m}\right) ^{m}T\left(
1\right) \},  \label{32}
\end{equation}%
where $M,$ $T$ are defined as in (\ref{34}). At the same time, this
inequality is the same as inequality (\ref{4}).
\end{corollary}

\begin{theorem}
Let $f:\left[ 0,\infty \right) \rightarrow \left( 0,\infty \right) $ be an $%
\left( \alpha ,m\right) $-logarithmically convex function on $\left[ 0,\frac{%
b}{m}\right] $ with $a,b\in \left[ 0,\infty \right) ,$ $0\leq a<b<\infty ,$ $%
\left( \alpha ,m\right) \in \left( 0,1\right] \times \left( 0,1\right] .$
Then%
\begin{equation}
\frac{1}{b-a}\int_{0}^{1}G\left( f\left( x\right) ,f\left( a+b-x\right)
\right) dx\leq \left[ f\left( \frac{a}{m}\right) f\left( \frac{b}{m}\right) %
\right] ^{m}S\left( \alpha \right)  \label{42}
\end{equation}%
where 
\begin{equation*}
\theta =f\left( a\right) f\left( b\right) \left[ f\left( \frac{a}{m}\right) %
\right] ^{-m}\left[ f\left( \frac{b}{m}\right) \right] ^{-m}
\end{equation*}%
and%
\begin{equation*}
S\left( \alpha \right) =\left\{ 
\begin{array}{c}
1,\text{ \ \ \ \ \ \ \ \ \ \ }\theta =1, \\ 
\frac{\theta ^{\alpha }-1}{\alpha \ln \theta },\text{ }0<\theta <1,%
\end{array}%
\right.
\end{equation*}%
$G\left( a,b\right) ,$ $L\left( a,b\right) $ are geometric, logarithmic mean
respectively.
\end{theorem}

\begin{proof}
Since $f$ is $m$-log-convex, we have that%
\begin{eqnarray}
f\left( ta+m\left( 1-t\right) \frac{b}{m}\right) &\leq &\left[ f\left(
a\right) \right] ^{t^{\alpha }}\left[ f\left( \frac{b}{m}\right) \right]
^{m\left( 1-t^{\alpha }\right) },  \label{q} \\
f\left( tb+m\left( 1-t\right) \frac{a}{m}\right) &\leq &\left[ f\left(
b\right) \right] ^{t^{\alpha }}\left[ f\left( \frac{a}{m}\right) \right]
^{m\left( 1-t^{\alpha }\right) }  \notag
\end{eqnarray}%
for all $t\in \left[ 0,1\right] $. If we multiply the above inequalities and
take square roots, we obtain%
\begin{eqnarray*}
G\left( f\left( ta+\left( 1-t\right) b\right) ,f\left( tb+\left( 1-t\right)
a\right) \right) &\leq &\left[ f\left( a\right) f\left( b\right) \right]
^{t^{\alpha }}\left[ f\left( \frac{a}{m}\right) f\left( \frac{b}{m}\right) %
\right] ^{m\left( 1-t^{\alpha }\right) } \\
&=&\left[ f\left( \frac{a}{m}\right) f\left( \frac{b}{m}\right) \right]
^{m}\theta ^{t^{\alpha }}
\end{eqnarray*}%
for all $t\in \left[ 0,1\right] $. Integrating this inequality on $[0,1]$
with respect to $t$, we obtain%
\begin{eqnarray*}
&&\int_{0}^{1}G\left( f\left( ta+\left( 1-t\right) b\right) ,f\left(
tb+\left( 1-t\right) a\right) \right) dt \\
&\leq &\left[ f\left( \frac{a}{m}\right) f\left( \frac{b}{m}\right) \right]
^{m}\int_{0}^{1}\theta ^{t^{\alpha }}dt
\end{eqnarray*}%
When $\theta =1$, we have $\int_{0}^{1}\theta ^{t^{\alpha }}dt=1.$ When $%
\theta <1$, we have%
\begin{equation*}
\int_{0}^{1}\theta ^{t^{\alpha }}dt\leq \int_{0}^{1}\theta ^{\alpha t}dt=%
\frac{\theta ^{\alpha }-1}{\alpha \ln \theta }.
\end{equation*}%
If we change the variable $x=ta+\left( 1-t\right) b$, $t\in \left[ 0,1\right]
$, we get%
\begin{equation*}
\int_{0}^{1}G\left( f\left( ta+\left( 1-t\right) b\right) ,f\left( tb+\left(
1-t\right) a\right) \right) dt=\frac{1}{b-a}\int_{0}^{1}G\left( f\left(
x\right) ,f\left( a+b-x\right) \right) dx
\end{equation*}%
and the inequality in (\ref{42}) is proved.
\end{proof}

\end{document}